\documentclass[12pt]{amsart}

\usepackage{txfonts}
\newtheorem{theorem}{Theorem}[section]
\newtheorem{proposition}[theorem]{Proposition}
\newtheorem{lemma}[theorem]{Lemma}

\newtheorem{corollary}[theorem]{Corollary}

\theoremstyle{definition}
\newtheorem{definition}[theorem]{Definition}
\newtheorem{example}[theorem]{Example}

\theoremstyle{remark}

\numberwithin{equation}{section}

\catcode`\ç=13
\defç{\c{c}}
\catcode`\é=13
\defé{\'e}
\catcode`\à=13
\defà{\`a}
\catcode`\è=13
\defè{\`e}
\catcode`\â=13
\defâ{\^a}
\catcode`\ù=13
\defù{\`u}
\catcode`\ê=13
\defê{\^e}
\catcode`\î=13
\defî{\^\i}
\catcode`\ô=13
\defô{\^o}
\newcommand{\field}[1]{\mathbb{#1}}

\newcommand{\Z }{\field{Z}}

\begin{document}

\title{Pr\"ufer property in amalgamated algebras along an ideal}

%    Information for first author
\author[Najib Mahdou]{Najib Mahdou}
\address{Department of Mathematics, Faculty of Science and Technology of Fez, Box 2202, University S. M.
Ben Abdellah Fez, Morocco}
 \email{mahdou@hotmail.com}

%    Current address
%\curraddr{Department of Mathematics, Faculty of Science and
%Technology of Fez,\\ Box 2202, University S. M. Ben Abdellah Fez,
%Morocco}
 \email{moutu\_2004@yahoo.fr}
%    \thanks will become a 1st page footnote.
%\thanks{The first author was supported in part by NSF Grant \#000000.}

%    Information for second author
\author[Moutu Abdou Salam Moutui]{Moutu Abdou Salam Moutui}
%    Address of record for the research reported here
\address{Department of Mathematics, Faculty of Science and Technology of Fez, Box 2202, University S. M.
Ben Abdellah Fez, Morocco}
%\thanks{Support information for the second author.}

%    General info

%\date{January 1, 2001 and, in revised form, June 22, 2001.}

%\dedicatory{Dedicated to our Advisor Salah-Eddine Kabbaj.}


\begin{abstract} Let $f : A \rightarrow B$ be a ring homomorphism and  $J$ be an ideal of $B$. In this paper, we give a
characterization of zero divisors of the amalgamation which is a
generalization of Maimani's and Yassemi's work (see \cite{y}).
Also, we investigate the transfer of Pr\"ufer domain concept to
commutative rings with zero divisors in the amalgamation of $A$
with $B$ along $J$ with respect to $f$ (denoted by $A\bowtie^fJ),$
introduced and studied by D'Anna, Finocchiaro and Fontana in 2009
(see \cite{AFF1} and \cite{AFF2}). Our aim is to provide new
classes of commutative rings satisfying this property.\smallskip
\end{abstract}

\maketitle

\section{Introduction} All rings considered in this paper are commutative with identity elements and all modules are unital. In 1932, Pr\"ufer introduced
and studied in \cite{P} integral domains in which every finitely
generated ideal is invertible. In 1936, Krull \cite{Kru} named
these rings after H. Pr\"ufer and stated equivalent conditions
that make a domain Pr\"ufer. Through the years, Pr\"ufer  domains
acquired a great many equivalent characterizations, each of which
was extended to rings with zero-divisors in different ways. In
their recent paper devoted to Gaussian properties, Bazzoni and
Glaz have proved that a Pr\"ufer ring satisfies any of the other
four Pr\"ufer conditions if and only if its total ring of
quotients satisfies that same condition \cite[Theorems 3.3 \& 3.6
\& 3.7 \& 3.12]{BG2}. The authors investigate in \cite{bm1} the
transfer of these Pr\"ufer-like properties to pullbacks, and then
generate new families of rings with zero divisors subject to some
given Pr\"ufer conditions. In \cite{bkm}, the authors examined the
transfer of the Pr\"ufer conditions and obtained further evidence
for the validity of Bazzoni-Glaz conjecture sustaining that "the
weak global dimension of a Gaussian ring is 0, 1, or $\infty$"
\cite{BG2}. Notice that both conjectures share the common context
of rings. Recall that classical examples of non-semi-hereditary
arithmetical rings stem from Jensen's 1966 result \cite{J} as
non-reduced principal rings, e.g., $\Z/n^{2}\Z$ for any integer
$n\geq 2$. Abuihlail, Jarrar and Kabbaj studied in \cite{abjk}
commutative rings in which every finitely generated ideal is
quasi-projective (denoted by $fqp$-rings). They proved that this
class of rings stands strictly between the two classes of
arithmetical rings and Gaussian rings. Thereby, they generalized
Osofsky's theorem on the weak global dimension of arithmetical
rings and partially resolve Bazzoni-Glaz$^{'}$s related conjecture
on Gaussian rings. They also established an analogue of
Bazzoni-Glaz results on the transfer of Pr\"ufer conditions
between a ring and its total ring of quotients. In \cite{CJKM},
the authors studied the transfer of the notions of local Pr\"ufer
ring and total ring of quotients. They examined the arithmetical,
Gaussian, fqp conditions to amalgamated duplication along an
ideal. At this point, we make the following definition:
\begin{definition}  Let $R$ be a commutative ring.\\
\begin{enumerate}
\item $R$ is called a \emph{semi-hereditary} if every finitely generated ideal of R is projective (see \cite{BS}).
\item $R$ is said to have  \emph{weak global dimension $\leq$ 1} if every finitely generated ideal of R is flat (see \cite{G2}).
\item $R$ is called an \emph{arithmetical ring} if the lattice formed by its ideals is distributive (see \cite{Fu}). \\
\item $R$ is called a \emph{Gaussian ring} if for every $f, g \in R[X]$, one has the content ideal equation $c(fg) = c(f)c(g)$ (see \cite{T}).\\
\item $R$ is called a \emph{Pr\"ufer ring} if every finitely generated regular ideal of $R$ is invertible (equivalently, every two-generated regular ideal is invertible), (See \cite{BS,Gr}).
\end{enumerate}
\end{definition}
In \cite{G3}, it is proved that each one of the above conditions implies the following
next one:  \begin{center}
Semi-hereditary $\Rightarrow$ weak global dimension$\leq$ 1 $\Rightarrow$ Arithmetical $\Rightarrow$ Gaussian $\Rightarrow$ Pr\"ufer.
\end{center}
Also examples are given to show that, in general, the implications cannot be reversed. Moreover, an investigation is carried
out to see which conditions may be added to any of these properties in order
to reverse the implications. Recall that in the domain context, the above five
class of Pr\"ufer-like rings collapse to the notion of Pr\"ufer domain. For more details on these notions, we refer to reader to \cite{bkm,bm1,BG,BG2,G2,G3,LR,Lu,T}.\\

In this paper, we investigate the transfer of Pr\"ufer property in amalgamation of rings issued from local rings, introduced and studied by D'Anna, Finocchiaro, and  Fontana in \cite{AFF1, AFF2} and defined  as follows :

\begin{definition}

Let $A$ and $B$ be two rings with unity, let $J$ be an ideal of
$B$ and let $f: A\rightarrow B$ be a ring homomorphism. In this
setting, we can consider the following subring of $A\times B$:
\begin{center} $A\bowtie^{f}J: =\{(a,f(a)+j)\mid a\in A,j\in
J\}$\end{center} called \emph{the amalgamation of $A$ and $B$
along $J$ with respect to $f$}. In particular, they have studied amalgmations in the frame of pullbacks which allowed them to establish
numerous (prime) ideal and ring-theoretic basic properties for this new construction. This
construction is a generalization of \emph{the amalgamated
duplication of a ring along an ideal} (introduced and studied by
D'Anna and Fontana in \cite{A, AF1, AF2}). The interest of amalgamation
resides, partly, in its ability to cover several basic constructions in commutative algebra,
including pullbacks and trivial ring extensions (also called Nagata's idealizations)(cf. \cite[page 2]{Nagata}). Moreover, other
classical constructions (such as the $A+XB[X]$, $A+XB[[X]]$, and the $D+M$ constructions) can be studied as particular cases of the amalgamation (\cite[Examples 2.5 and 2.6]{AFF1}) and other classical constructions, such as the CPI extensions (in the sense of Boisen and Sheldon \cite{Boisen}) are strictly related to it (\cite[Example 2.7 and Remark 2.8]{AFF1}). In \cite{AFF1}, the authors studied the basic properties of this construction (e.g., characterizations for $A\bowtie^{f}J$ to be a Noetherian ring, an integral domain, a reduced ring) and they characterized those distinguished pullbacks that can be expressed as an amalgamation. Moreover, in \cite{AFF2}, they pursued the investigation on the structure of the rings of the form $A\bowtie^{f}J$, with particular attention to the prime spectrum, to the chain properties and to the Krull dimension.
\end{definition}

\section{Main results}\label{sec:2}
We start by giving a description of the set of zero-divisors of
$A\bowtie^{f}J$. Recall that $Z(R)$ is the set of zero-divisors of
a ring $R$ and $Reg(R)$ is the set of regular elements of $R$.
Recall that an element $m$ of a module $M$ over a ring $R$ is
called a torsion element of the module if there exists a regular
element $r \in R$ that annihilates
$m$, i.e., $r m = 0$. A module $M$ over a ring $R$ is called a torsion module if all its elements are torsion elements.\\
\begin{proposition}\label{prop2}
Let $(A,B)$ be a pair of rings, $f : A \rightarrow B$ be a ring homomorphism and $J$ be a  proper ideal of $B$. Assume that at least one of the following conditions hold :\\
$(a)$ $J\subset f(A)$.\\
$(b)$ $J$ is a torsion $A-$module (with the $A-$module structure inherited by $f$).\\
$(c)$ $J^2=0$. \\

Then $Z(A\bowtie^{f}J)=\{Z(A)\bowtie^{f}J$ $\}\cup \{(a,0)/f(a)\in J\}\cup \{(0,x)/x\in J\}\cup \{(a,f(a)+i)/$ $a$ is a regular element of $A$ and $(f(a)+i)j=0,$ for some $0\neq j\in J\}$.\\

\end{proposition}
\begin{proof}
$(1)$ Assume that $J\subseteq f(A)$. It is clear that
$\{(a,0)/f(a)\in J\}\subseteq Z(A\bowtie^{f}J),$  $\{(0,x)/x\in
J\}\subseteq Z(A\bowtie^{f}J)$ and $\{(a,f(a)+i)/$ $a$ is a
regular element and $(f(a)+i)j=0$ for some $0\neq j\in
J\}\subseteq Z(A\bowtie^{f}J)$. It remains to show that
$Z(A)\bowtie^{f}J\subseteq Z(A\bowtie^{f}J)$. Let $(a,f(a)+i)\in
Z(A)\bowtie^{f}J.$ Clearly, $a\in Z(A)$ and $i\in J.$ So, there
exist $0\neq b\in A$ such that $ab=0$. If $i=0$ or $f(b)=0$ or
$f(b)\in Ann(J),$ then there exists $(0,0)\neq (b,f(b))\in
A\bowtie^{f}J$ such that $(a,f(a)+i)(b,f(b))=(0,0)$ and so
$(a,f(a)+i)\in Z(A\bowtie^{f}J)$. Assume that $0\neq i,$ $f(b)\neq
0$ and $f(b)\not \in Ann(J).$ Then, there exists $0\neq k\in J$
such that $f(b)k\neq 0$. Using the fact $J\subseteq f(A),$ there
exists some $x\in f^{-1}(J)$ such that $f(x)=k$ and so clearly
$bx\neq 0$ and  $(bx,0)\in A\bowtie^{f}J$. Therefore, there exists
$(0,0)\neq(bx,0)\in A\bowtie^{f}J$ such that
$(a,f(a)+i)(bx,0)=(0,0).$ Hence, $(a,f(a)+i)\in Z(A\bowtie^{f}J).$
Thus, $Z(A)\bowtie^{f}J\subseteq Z(A\bowtie^{f}J)$. Conversely,
Let $(a,f(a)+i)\in Z(A\bowtie^{f}J)$. If $f(a)=-i$ or $a=0,$ then
$(a,f(a)+i)\in \{(a,0)/f(a)\in J\}$ or $(a,f(a)+i)\in \{(0,x)/x\in
J\}$. Assume that $f(a)\neq-i$ and $a\neq0$. Then, there exists
$(0,0)\neq(b,f(b)+j)\in A\bowtie^{f}J$ such that
$(a,f(a)+i)(b,f(b)+j)=(0,0)$. Therefore, $ab=0$ and
$(f(a)+i)(f(b)+j)=0$. If $a$ is a regular element, then $b=0,$ and
$j(f(a)+i)=0$. So, $(a,f(a)+i)\in \{(a,f(a)+i)/$ $a$ is a regular
element and $j(f(a)+i)=0$ for some $0\neq j\in J\}$. Assume that
$a$ is a zero-divisor of $A$ $(b\neq0)$. Then, $a\in Z(A)$ and
$i\in J$. And so $(a,f(a)+i)\in \{Z(A)\bowtie^{f}J\}$. Hence,
$Z(A\bowtie^{f}J)=\{Z(A)\bowtie^{f}J$ $\}\cup \{(a,0)/f(a)\in
J\}\cup \{(0,x)/x\in J\}\cup \{(a,f(a)+i)/$ $a$ is a regular
element and
$(f(a)+i)j=0$ for some $0\neq j\in J\},$ as desired.\\
$(2)$ Assume that $J$ is a torsion $A-$module. Similar arguments
as $(1)$ above, we show that $Z(A\bowtie^{f}J)\subseteq
\{Z(A)\bowtie^{f}J$ $\}\cup \{(a,0)/f(a)\in J\}\cup \{(0,x)/x\in
J\}\cup \{(a,f(a)+i)/$ $a$ is a regular element of $A$ and
$(f(a)+i)j=0,$ for some $0\neq j\in J\}$. On the other hand, one
can easily check that $\{(a,0)/f(a)\in J\}\subseteq
Z(A\bowtie^{f}J),$ $\{(0,x)/x\in J\}\subseteq Z(A\bowtie^{f}J)$
and $\{(a,f(a)+i)/$ $a$ is a regular element and $(f(a)+i)j=0$ for
some $0\neq j\in J\}\subseteq Z(A\bowtie^{f}J)$. It remains to
show that $\{Z(A)\bowtie^{f}J$ $\}\subseteq Z(A\bowtie^{f}J)$. Let
$(a,f(a)+i)\in Z(A)\bowtie^{f}J.$ Clearly, $a\in Z(A)$ and $i\in
J.$ So, there exist $0\neq b\in A$ such that $ab=0$. Using the
fact $J$ is a torsion $A-$module, there is some $x\in A\backslash
Z(A)$ such that $f(x)i=0$ and it is obvious that $0\neq xb$ since
$0\neq b$. So, there exists $(0,0)\neq(xb,f(xb))\in A\bowtie^{f}J$
such that $(xb,f(xb))(a,f(a)+i)=(0,0)$. Hence, $(a,f(a)+i)\in
Z(A\bowtie^{f}J)$. Thus,
$\{Z(A)\bowtie^{f}J$ $\}\subseteq Z(A\bowtie^{f}J)$, as desired.\\
$(3)$ Assume that $J^2=0$. It is clear that $\{(a,0)/f(a)\in J\}\subseteq Z(A\bowtie^{f}J),$  $\{(0,x)/x\in J\}\subseteq Z(A\bowtie^{f}J)$
and $\{(a,f(a)+i)/$ $a$ is a regular element and $(f(a)+i)j=0$ for some $0\neq j\in J\}\subseteq Z(A\bowtie^{f}J)$. It remains to show that
$Z(A)\bowtie^{f}J\subseteq Z(A\bowtie^{f}J)$. Let $(a,f(a)+i)\in Z(A)\bowtie^{f}J.$ Clearly, $a\in Z(A)$ and $i\in J.$ So, there exist $0\neq b\in A$
such that $ab=0$. If $i=0$ or $f(b)=0$ or $f(b)\in Ann(J),$ then there exists $(0,0)\neq (b,f(b))\in A\bowtie^{f}J$ such that $(a,f(a)+i)(b,f(b))=(0,0)$
and so $(a,f(a)+i)\in Z(A\bowtie^{f}J)$. Assume that $0\neq i,$ $f(b)\neq 0$ and $f(b)\not \in Ann(J).$ Then, there exists $0\neq k\in J$ such that
$f(b)k\neq 0$. One can easily check that $(a,f(a)+i)(b,f(b)+kf(b))=(0,0)$. So, there exists $(0,0)\neq(b,f(b)+kf(b)\in A\bowtie^{f}J$ such that
$(a,f(a)+i)(b,f(b)+kf(b)=(0,0)$. Hence, $(a,f(a)+i)\in Z(A\bowtie^{f}J)$ and $Z(A)\bowtie^{f}J\subseteq Z(A\bowtie^{f}J)$. Conversely, with similar
arguments as $(1)$, we obtain that $Z(A\bowtie^{f}J)\subseteq \{Z(A)\bowtie^{f}J$ $\}\cup \{(a,0)/f(a)\in J\}\cup \{(0,x)/x\in J\}\cup \{(a,f(a)+i)/$ $a$
is a regular element of $A$ and $(f(a)+i)j=0,$ for some $0\neq j\in J\}$, as desired.
\end{proof}
It is worthwhile noting that the statement $(1)$ of Proposition \ref{prop2} recovers \cite[Proposition 2.2]{y} where $f$ is the identity maps.\\

The main Theorem of this paper develops a result on the transfer of the Pr\"ufer property to amalgamation of rings issued from local rings.
\begin{theorem}\label{thm0} Let $(A,m)$ be a local ring, $B$ be a ring, $f : A \rightarrow B$ be a ring homomorphism and $J$ be a proper ideal of $B$ such that $J\subseteq Rad(B)$. Assume that $J\subset f(A)$ and $f(Reg(A))= Reg(B)$. Then $A\bowtie^{f}J$ is Pr\"ufer if and only if so is $A$ and $J=f(a)J$ for all $a\in m \backslash Z(A)$.\\
\end{theorem}
The proof of this theorem involves the following lemmas. Recall that $V_B(J):=\{P\in$ $ $Spec(B)$: P\supseteq J\}$ (for more details see \cite{AFF2}). We denote $V_B(J)$ simply by $V(J)$.

\begin{lemma}\label{lem1}
Let $(A,B)$ be a pair of rings, $f : A \rightarrow B$ be a ring homomorphism and  $J$ be a proper ideal of
$B$. Then, $A\bowtie^{f}J$ is local if and only if so is $A$ and $J\subseteq Rad(B)$.
\end{lemma}
\begin{proof}
By \cite[Proposition 2.6 (5)]{AFF2}, $Max(A\bowtie^{f}J)=\{m\bowtie^{f}J$ / $m\in Max(A)\}\cup\{\overline{Q}^f\}$ with $Q\in Max(B)$ not containing $V(J)$ and $\overline{Q}^f:=\{(a,f(a)+j)$ / $a\in A, j\in J$ and $f(a)+j \in Q$ \}.
Assume that $A\bowtie^{f}J$ is local. It is clear that $A$ is local by the above characterization of $Max(A\bowtie^{f}J)$. We claim that $J\subseteq Rad(B)$. Deny. Then there exist $Q\in Max(B)$ not containing $V(J)$ and so $Max(A\bowtie^{f}J)$ contains at least two maximal ideals, a contradiction since $A\bowtie^{f}J$ is local. Hence, $J\subseteq Rad(B)$. Conversely, assume that $(A,m)$ is local and $J\subseteq Rad(B)$. Then $J$ is contained in $Q$ for all $Q\in Max(B)$. Consequently, the set $\{\overline{Q}^f\}$ is empty. And so $Max(A\bowtie^{f}J)=\{m\bowtie^{f}J/m\in Max(A)\}$. Hence, $m\bowtie^{f}J$ is the only maximal ideal of $A\bowtie^{f}J$ since $(A,m)$ is local. Thus, $(A\bowtie^{f}J,M)$ is local with $M=m\bowtie^{f}J$, as desired.
\end{proof}

Recall that a polynomial $f$ is called Gaussian over a ring $A$ if for any polynomial $h$ over $A$, $c(fh) = c(f)c(h)$ holds, where $c(f)$ is the content of $f$ (the ideal generated by the coefficients of the polynomial $f(x)$). For more details of this notion, we refer to reader to see \cite{T,c,HH}.
\begin{lemma}\label{lem3}
Let $(A,B)$ be a pair of rings, $f : A \rightarrow B$ be a ring homomorphism and  $J$ be a  proper ideal of $B.$ If the polynomial $F:=\sum_{i=0}^{n}(a_i,f(a_i))x^i$ is Gaussian over $A\bowtie^{f}J$, then $f:=\sum_{i=0}^{n}a_ix^i$ is Gaussian over $A$.
\end{lemma}
\begin{proof}
Assume that $F:=\sum_{i=0}^{n}(a_i,f(a_i))x^i$ is Gaussian over $A\bowtie^{f}J$. We claim that $f:=\sum_{i=0}^{n}a_ix^i$ is Gaussian over $A$. Indeed, let $h(x):=\sum_{j=0}^{m}b_jx^j$ be a polynomial over $A$. Then, $H(x):=\sum_{j=0}^{m}(b_j,f(b_j))x^j$ is a polynomial over $A\bowtie^{f}J$. Consider $\gamma \in c(f)c(h).$ Clearly, $(\gamma,f(\gamma))\in c(F)c(H)=c(FH)=\sum_{i+j=k}^{}(a_ib_j,f(a_i)f(b_j))$ with $k\in \{0,1,...,m+n\}$ since $F$ is Gaussian in $A\bowtie^{f}J$. Consequently, $\gamma\in \sum_{i+j=k}^{}a_ib_j=c(fh)$. Finally, $f:=\sum_{i=0}^{n}a_ix^i$ is Gaussian over $A$.
\end{proof}
\begin{lemma}\label{lem4}
Let $(A,B)$ be a pair of rings, $f : A \rightarrow B$ be a ring homomorphism and  $J$ be a  proper ideal of $B$. Assume that $f(Reg(A))\subseteq Reg(B)$. If $A\bowtie^{f}J$ is Pr\"ufer, then so is $A$.\\
\end{lemma}
\begin{proof}
Assume that $\forall a\in A\backslash Z(A),$ $f(a)\in B\backslash Z(B)$ and $A\bowtie^{f}J$ is Pr\"ufer. Let $I:=\sum_{i=0}^{n}Aa_i$ be a finitely generated regular ideal of $A$ and $a$ be a regular element of $I$. Clearly, $H:=\sum_{i=0}^{n}A\bowtie^{f}J(a_i,f(a_i))$ is a finitely generated regular ideal of $A\bowtie^{f}J$ since $(a,f(a))$ is a regular element of $H$. So, $H$ is invertible. Therefore, the polynomial $F(x):=\sum_{i=0}^{n}(a_i,f(a_i))x^i$ is Gaussian over $A\bowtie^{f}J$. By Lemma \ref{lem3}, $f(x):=\sum_{i=0}^{n}a_ix^i$ is Gaussian over $A$. Hence, by \cite[Theorem 4.2 (2)]{BG}, $I=c(f)$ is invertible. Thus, $A$ is Pr\"ufer, as desired.\\
\end{proof}

\begin{lemma}\label{lem5}
Let $(A,m)$ be a local Pr\"ufer ring, $f:A\rightarrow B$ be a ring homomorphism and $J$ be a proper ideal of $B$ such that $J\subseteq Rad(B)$. Assume that $J\subset f(A),$ $f^{-1}(J)\subset Z(A)$ and $f(Reg(A))\subseteq Reg(B)$. Then $Z(A\bowtie^{f}J)=Z(A)\bowtie^{f}J$.\\
 \end{lemma}
\begin{proof}
 Assume that $J\subset f(A),$ $f^{-1}(J)\subset Z(A)$ and $f(Reg(A))\subseteq Reg(B)$. By Proposition \ref{prop2}, $Z(A\bowtie^{f}J)=\{Z(A)\bowtie^{f}J$ $\}\cup \{(a,0)/f(a)\in J\}\cup \{(0,x)/x\in J\}\cup \{(a,f(a)+i)/$ $a$ is a regular element of $A$ and $(f(a)+i)j=0,$ for some $0\neq j\in J\}$. It is easy to see that the set $\{(a,0)/f(a)\in J\}=f^{-1}(J)\times \{0\}\subset Z(A)\bowtie^{f}J$ since $f^{-1}(J)\subset Z(A)$ and one can easily check that $\{(0,x)/x\in J\}\subset Z(A)\bowtie^{f}J$. Our aim is to show that if $A$ is Pr\"ufer, the set $\{(a,f(a)+i)/$ a is a regular element of $A$ and $(f(a)+i)j=0,$ for some $0\neq j\in J\}$ is empty. Indeed, we show that $f(a)+i$ is a regular element of $B$. Let $a\in A\backslash Z(A)$ and let $i\in J$. Using the fact $J\subset f(A),$ there exists $x\in f^{-1}(J)$ such that $i=f(x)$ and so $f(a)+i=f(a)+f(x)=f(a+x)$. Since $f(Reg(A))\subseteq Reg(B)$, then it suffices to show that $a+x$ is a regular element of $A$. Since the ideal $(a,x)$ is a regular ideal,then $(a)$ and $(x)$ are comparable by \cite[Lemma 3.8]{abjk} since $A$ is Pr\"ufer and we have necessarily $x=ak$ for some $k$ non-unit in $A$. One can easily check that $a+x=(1+k)a$ which is a regular element of $A$. Consequently, $f(a)+i$ is a regular element of $B$ and it follows that the set $\{(a,f(a)+i)/$ $a$ is a regular element of $A$ and $(f(a)+i)j=0,$ for some $0\neq j\in J\}$ is empty. Hence, $Z(A\bowtie^{f}J)=Z(A)\bowtie^{f}J$.
\end{proof}
\begin{proof}[Proof of Theorem \ref{thm0}]
By Lemma \ref{lem1}, $(A\bowtie^{f}J,m\bowtie^{f}J)$ is local.\\
   Assume that $J\subset f(A),$ $f(Reg(A))=Reg(B)$ and $A\bowtie^{f}J$ is Pr\"ufer. By Lemma \ref{lem4}, $A$ is Pr\"ufer. Before proving that $f^{-1}(J)\subset Z(A)$, we show that if $A\bowtie^{f}J$ is Pr\"ufer, then $J\subseteq Z(B).$ Deny. Let $j\in J$ such that $j\in Reg(B)$. Since $J\subset f(A),$ there exists $x\in Reg(A)$ such that $j=f(x)$ since $f(Reg(A))=Reg(B)$ and so $(x,j)\in Reg(A\bowtie^{f}J)$. By \cite[Lemma 3.8]{abjk}, $((0,j))$ and $((x,j))$ are comparable since $A\bowtie^{f}J$ is Pr\"ufer. And necessarily, $(0,j)=(x,j)(b,f(b)+k)$. Therefore, $b=0$ and $kj=j.$ So, $j(1-k)=0$. It follows that $j=0,$ a contradiction since $j\in Reg(B).$ Hence, $J\subseteq Z(B).$ Next, we claim that $f^{-1}(J)\subset Z(A)$. Deny. There exists $x\in f^{-1}(J)$ such that $x\not \in Z(A)$. Then, $f(x)\in reg(B)$ since $f(Reg(A))=Reg(B)$, a contradiction since $J\subset Z(B)$. Hence, $f^{-1}(J)\subset Z(A)$. By Lemma \ref{lem5}, $Z(A\bowtie^{f}J)=Z(A)\bowtie^{f}J$. Now, we claim that $f(a)J=J$ for all $a\in m\backslash Z(A)$. Indeed, it is clear that $f(a)J\subseteq J$ $\forall$ $a\in m \backslash Z(A)$. Conversely, let $a\in m\backslash Z(A)$ and let $i\in J.$ Clearly, $(a,f(a)+i)\in (A\bowtie^{f}J)\backslash Z(A\bowtie^{f}J)$ since $Z(A\bowtie^{f}J)=Z(A)\bowtie^{f}J$ and we have $(0,i)\in A\bowtie^{f}J$. By \cite[Lemma 3.8]{abjk}, the ideal $(a,f(a)+i)$ and $(0,i)$ are comparable since $A\bowtie^{f}J$ is Pr\"ufer and necessarily $(0,i)=(a,f(a)+i)(b,f(b)+j)$ for some $(b,f(b)+j)\in A\bowtie^{f}J$. So, $b=0$ and $i=(f(a)+i)j$ for some $j\in J.$ Hence, it follows that $i=f(a)j(1-j)^{-1}$ for some $0\neq j\in J$ and so $i\in f(a)J$. Thus, $J=f(a)J$. Conversely, assume that $A$ is Pr\"ufer and $J=f(a)J$ for all $a\in m \backslash Z(A)$. We show that $J=f(a)J$ for all $a\in m \backslash Z(A)$ implies that $J\subset Z(B)$. Deny. Let $j\in J$ such that $j\in Reg(B)$. Using the fact $J\subset f(A),$ $j=f(x)$ for some $x\in Reg(A)$. So, $j=f(x)l$ for some $l\in J$. It follows that $j(1-l)=0$ and so $j=0$ since $(1-l)$ is invertible in $B$ ($l\in J\subseteq Rad(B)).$ Hence, $J\subset Z(B)$ and so one can easily check that $J\subset Z(B)$ implies that $f^{-1}(J)\subset Z(A)$. By Lemma \ref{lem5}, $Z(A\bowtie^{f}J)=Z(A)\bowtie^{f}J$. Next, we show that $A\bowtie^{f}J$ is Pr\"ufer. Consider $F:=((a,f(a)+i),(b,f(b)+j))$ be a regular ideal of $A\bowtie^{f}J$. Assume one, at least, of the two generators of $F$ is regular. By Lemma \ref{lem5}, $(a)$ or $(b)$ is a regular ideal of $A$ and by \cite[Lemma 3.8]{abjk}, $(a)$ and $(b)$ are comparable in $A$. We may assume that $(a)$ is regular and $b=ax$ for some $x\in A$. Using the fact $J=f(a)J$ for all $a\in m \backslash Z(A)$, there is $k\in J$ such that $j-if(x)=(f(a)+i)k$ since there is some $l\in J$ such that $f(a)+i=f(a)(1+l)$ which is a regular element and $f(a)+i=f(a+c)$ for some $c\in f^{-1}(J)$. One can easily check that $(b,f(b)+j)=(a,f(a)+i)(x,f(x)+k)$. Consequently, $F:=((a,f(a)+i))$. Now, we assume that both generators of $F$ are zero divisors and let $(y,f(y)+h)$ be a regular element of $A\bowtie^{f}J$. Then, with similar arguments as above, there exist $a',b'\in A$ and $k_1,k_2\in J$ such that $a=a'y,$ $b=b'y,$ and $i-hf(a')=(f(y)+h)k_1,$ $j-hf(b')=(f(y)+h)k_2$. Therefore, $(a,f(a)+i)=(a',f(a')+k_1)(y,f(y)+h)$ and $(b,f(b)+j)=(b',f(b')+k_2)(y,f(y)+h)$. Consequently, $F:=((y,f(y)+h))$. Hence, $F$ is principal and so invertible. Thus, by \cite[Theorem 2.13 (2)]{BG}, $A\bowtie^{f}J$ is Pr\"ufer. \\
\end{proof}
The following corollary is a consequence of Theorem \ref{thm0}.\\
\begin{corollary}\label{c}
Let $(A,m)$ be a local ring and $I$ be a proper ideal of $A$. Then $A\bowtie I$ is Pr\"ufer if and only if so is $A$ and $I=aI$ for all $a\in m\backslash Z(A)$.
\end{corollary}
\begin{proof}
It is easy to see that $A\bowtie I=A\bowtie^f J$ where $f$ is the identity map of $A$, $B=A,$ $J=I$. By Lemma \ref{lem1}, $(A\bowtie I,M)$ is a local ring with $M=m\bowtie I$ since $(A,m)$ is a local ring and $I\subseteq Rad(A)=m$. Moreover, $f(Reg(A))=Reg(A)=reg(B)$ and $J=I\subset f(A)=A$. By Theorem \ref{thm0}, the conclusion is straightforward.
 \end{proof}
The following example illustrates the failure of Theorem 2.3 in general, beyond the context $J\subset f(A)$ and $f(Reg(A))=Reg(B)$.\\
\begin{example}
Let $(A,m)$ be a non-valuation local domain, $E:=\frac{A}{m}$ and $B:=A\propto E$ be the trivial ring extension of $A$ by $E$. Consider  $$\begin{array}{clcl}
  f: & A & \hookrightarrow & B \\
   & a & \hookrightarrow & f(a)=(a,0) \\
\end{array}$$ be a ring homomorphism and $J:=0\propto E$ be a proper ideal of $B$. Then the following statement hold :\\
$(1)$ $J\not \subset f(A)$.\\
$(2)$ $f(Reg(A))\neq Reg(B)$.\\
$(3)$ $A$ is not Pr\"ufer.\\
$(4)$ $A\bowtie^{f}J$ is Pr\"ufer.\\
\end{example}
\begin{proof}
$(1)$ $f(A)=A\propto 0$ and $J=0\propto E.$ It is easy to see that $f(A)\cap J=(0)$. Hence, $J\not \subset f(A)$.\\
$(2)$ Let $0\neq b\in m.$ Then, $f(b)=(b,0)$ is a zero-divisor of $B$ since $(b,0)(0,e)=(0,0)$ for all $0\neq e\in E.$ Hence, $f(Reg(A))\neq Reg(B)$.\\
$(3)$ Straightforward.\\
$(4)$ One can easily check that $A\bowtie^{f}J=A\propto B$ which is Pr\"ufer by \cite[Example 2.8]{bkm}.
 \end{proof}
It is easy to see that if $A\bowtie^{f}J$ is Gaussian, then so is $A$ since the Gaussian property is stable under factor rings and by \cite[Proposition 5.1 (3)]{AFF1}, $A\cong \frac{A\bowtie^{f}J}{\{0\}\times \{J\}}$. Recall also that a local ring is Gaussian if "for any two elements $a,b$ in the ring, we have $<a,b>^2=<a^2>$ or $<b^2>$; moreover, if $ab=0$ and, say, $<a,b>^2=<a^2>$, then $b^2=0$" (see \cite[Theorem 2.2 ]{BG2}).\\

 Theorem \ref{thm0} enriches the literature with new examples of non-Gaussian Pr\"ufer rings.\\
 \begin{example}
Let $(B,m)$ be a local Pr\"ufer ring and $I$ be a proper ideal of $B$ such that $I=m=Z(B)$ with $I^2\neq (0),$ (for instance $B:=\frac{\mathbb{Z}}{8\mathbb{Z}}$ and $I=m:=2\frac{\mathbb{Z}}{8\mathbb{Z}}=Z(B))$. By Corollary \ref{c}, $A:=B\bowtie I$ is a local Pr\"ufer ring. Consider $f: A\rightarrow B$ be a surjective ring homomorphism and $J:=I$ be a proper ideal of $B$. Then :\\
$(1)$ $A\bowtie^{f}J$ is Pr\"ufer.\\
$(2)$ $A\bowtie^{f}J$ is not Gaussian.
\end{example}
\begin{proof}
$(1)$ It is easy to see that $J=I\subset f(A)=B$ and $Z(B\bowtie I)=Z(B)\bowtie I$ by Lemma \ref{lem5} since $(B,m)$ is (local) Pr\"ufer and $I\subseteq Z(B)$. One can easily check that $f(Reg(A))=Reg(B)$ and $J:=Z(B)$. So, by Theorem \ref{thm0}, $A\bowtie^{f}J$ is Pr\"ufer.\\
$(2)$ We claim that $A:=B\bowtie I$ is not Gaussian. Deny. $B$ is (local) Gaussian. Let $a,b\in I$. Then, $(a,0)$ and $(0,b)\in B\bowtie I$ (which is local Gaussian) and by \cite[Theorem 2.2 ]{BG2}, $<(a,0),(0,b)>^2=<(a^2,0),(0,b^2)>=<(a^2,0)>$ or $<(0,b^2)>$. Therefore, it follows that $b^2=0$ or $a^2=0$. So, using the fact $B$ is (local) Gaussian, $<a,b>^2=<a^2>=<0>$ or $<b^2>=<0>$ and so $ab=0$. It follows that $I^2=(0)$, a contradiction. Hence, $A$ is not Gaussian. Thus, $A\bowtie^{f}J$ is not Gaussian.
\end{proof}

 \begin{example}
Let $(A,m)$ be a non-Gaussian local Pr\"ufer ring such that $m:=Z(A)$ (for instance $(A,m):=(A_0\propto B,m_0\propto B)$ with $A_0$ be a non-valuation local domain and $B:=A_0\propto E$ with $E$ be a non-zero $A-$module such that $ME=0$ see \cite[Example 2.8]{bkm}). Consider $I:=m$ be a proper ideal of $A$. Then :\\
$(1)$ $A\bowtie I$ is Pr\"ufer.\\
$(2)$ $A\bowtie I$ is not Gaussian.\\
\end{example}
\begin{proof}
$(1)$ By Corollary \ref{c}, $A\bowtie I$ is Pr\"ufer.\\
$(2)$ $A\bowtie I$ is not Gaussian since $A$ is not Gaussian.\\
\end{proof}

\bibliographystyle{amsplain}

\begin{thebibliography}{10}
\bibitem{abjk} J. Abuihlail, M. Jarrar and S. Kabbaj, \textit{Commutative rings in which every finitely generated ideal is quasiprojective},
J. Pure Appl. Algebra 215 (2011) 2504-2511.

%\bibitem{a1} D.D. Anderson, \textit{Commutative rings}, in : Jim Brewer, Sarah Glaz, William Heinzer, Bruce Olberding (Eds.), Multiplicative Ideal Theory in Commutative Algebra: A tribute to the work of Robert Gilmer, Springer, New York, 2006, pp. 1-20.
%\bibitem{and}D. D. Anderson and B. J. Kang, \textit{Content formulas for polynomials and power series and complete integral closure}, J. Algebra  181 (1996) 82--94.\par
\bibitem{and1} D.D. Anderson, \textit{GCD domains, Gauss' Lemma, and contents of polynomials}, Non-Noetherian commutative ring theory, Math. Appl. Kluwer Acad. Publ., Dordrecht 520 (2000) 1-31.
%\bibitem{and2}D.D. Anderson, \textit{Another generalization of principal ideal rings}, J. Algebra 48 (1997), 409-416.
%\bibitem{and3} D.D. Anderson and V. Camillo, \textit{Armendariz rings and Gaussian rings}, Comm. Algebra 26 (1998), 2265-2272.
\bibitem{bkm} C. Bakkari, S. Kabbaj and N. Mahdou, \textit{Trivial extensions defined by $Pr\ddot{u}fer$ conditions}, J. of Pure Appl. Algebra 214 (2010) 53-60.
\bibitem{bm1}C. Bakkari and N. Mahdou, \textit{Pr\"ufer-like conditions in pullbacks}, Commutative algebra and its applications, Walter de Gruyter, Berlin, (2009) 41-47.
\bibitem{bmm}C. Bakkari, N. Mahdou and H. Mouanis, \textit{Pr\"ufer-like Conditions in Subrings Retract and Applications}, Comm. Algebra 37 (2009) 47-55.
\bibitem{BG} S. Bazzoni and S. Glaz, \textit{Pr\"ufer rings}, Multiplicative Ideal Theory in Commutative Algebra, Springer, New York, (2006) 55-72.
\bibitem{BG2} S. Bazzoni and S. Glaz, \textit{Gaussian properties of total rings of quotients}, J. Algebra 310 (2007) 180-193.\par
%\bibitem{bois}  M. Boisen and P. Sheldon, \textit{A note on pre-arithmetical rings}, Acta Math. Acad. Sci. Hungar. 28 (1976), 257–259.
\bibitem{Boisen}M. Boisen and  P.B. Sheldon; \textit{CPI-extension: Over rings of integral domains with special prime spectrum}, Canad. J. Math. {\bf 29} (1977) 722-737.
\bibitem{BS} H. S. Butts and W. Smith, \textit{Pr\"ufer rings}, Math. Z. 95 (1967) 196-211.
\bibitem{car} H. Cartan, S. Eilenberg, \textit{Homological Algebra}, Princeton University Press, 1956.
\bibitem{CJKM}    M. Chhiti, M. Jarrar, S. Kabbaj and N. Mahdou, \textit{$Pr\ddot{u}fer$ conditions in an amalgamated duplication
of a ring along an ideal}, Comm. Algebra, Accepted for
publication.\par
\bibitem{c}A. Corso and S. Glaz, \textit{Gaussian ideals and the Dedekind-Mertens Lemma}, Marcel Dekker Lecture Notes Pure Appl. Math 217 (2001) 131 - 143.
%\bibitem{do}J.L. Dorroh, \textit{Concerning adjunctions to algebras}, Bull. Amer. Math. Soc. 38 (1932), 85-88.
\bibitem{AFF1} M. D'Anna, C. A. Finocchiaro and M. Fontana, \textit{Amalgamated algebras along an ideal},
Commutative algebra and its applications, Walter de Gruyter, Berlin, (2009) 241-252.
 \bibitem{AFF2} M. D'Anna, C. A. Finocchiaro and M. Fontana, \textit{Properties
 of chains of prime ideals in amalgamated algebras along an ideal}, J. Pure Appl. Algebra 214 (2010) 1633-1641.
\bibitem{A} M. D'Anna, \textit{A construction of Gorenstein rings}, J. Algebra 306 (2006) 507-519.
\bibitem{AF1} M. D'Anna and M. Fontana, \textit{The amalgamated duplication of a ring along a multiplicative-canonical ideal}, Ark. Mat. 45 (2007) 241-252.
\bibitem{AF2} M. D'Anna and M. Fontana, \textit{An amalgamated duplication of a ring along an ideal : the basic properties}, J. Algebra Appl. 6 (2007) 443-459.

\bibitem{Fu} L. Fuchs, \textit{Uber die Ideale arithmetischer Ringe}, Comment. Math. Helv. 23 (1949) 334-341.\par
\bibitem{FHO1}  L. Fuchs, W. Heinzer and B. Olberding, \textit{Commutative ideal theory without finiteness conditions: Primal ideals}, Trans. Amer. Math. Soc. 357 (2005) 2771-2798.\par
\bibitem{G2}    S. Glaz, \textit{The weak global dimension of Gaussian rings}, Proc. Amer. Math. Soc. 133 (9) (2005) 2507-2513.\par
\bibitem{G3}    S. Glaz, \textit{Pr\"ufer conditions in rings with zero-divisors}, CRC Press Series of Lectures in Pure Appl. Math. 241 (2005) 272-282.\par
\bibitem{GV}    S. Glaz and W. Vasconcelos, \textit{The content of Gaussian polynomials}, J. Algebra 202 (1998) 1-9.\par
%\bibitem{gl}    S. Glaz and W. Vasconcelos, \textit{Gaussian polynomials}, Marcel Dekker Lecture Notes 186 (1997), 325-337.
\bibitem{Gr}    M. Griffin, \textit{Pr\"ufer rings with zero-divisors}, J. Reine Angew Math. 239/240 (1969) 55-67.\par
\bibitem{HH}    W. Heinzer and C. Huneke, \textit{Gaussian polynomials and content ideals}, Proc. Amer. Math. Soc. 125 (1997) 739-745.\par
\bibitem{H}     J. A. Huckaba, \textit{Commutative Rings with Zero-Divisors}, Marcel Dekker, New York, 1988.\par
\bibitem{J}     C. U. Jensen, \textit{Arithmetical rings}, Acta Math. Hungr. 17 (1966) 115-123.\par
\bibitem{Kru} W. Krull, \textit{Breitrage zur arithmetik kommutativer integritatsebereiche} Maths. Z 41 (1936) 545-577.
\bibitem{LR}    K. A. Loper and M. Roitman, \textit{The content of a Gaussian polynomial is invertible}, Proc. Amer. Math. Soc. 133 (5).
\bibitem{Lu} T. G. Lucas, \textit{Gaussian polynomials and invertibility}, Proc. Amer. Math. Soc. 133 (7) (2005) 1881-1886.\par

\bibitem{y} H. Maimani and S. Yassemi, \textit{Zero-divisor graphs of amalgamated duplication of a ring along an ideal}, J. Pure Appl. Algebra 212 (1) (2008) 168–174.

%\bibitem{NN}    B. Nashier and W. Nichols, A note on perfect rings, Manuscripta Math. 70 (3) (1991) 307--310.\par
\bibitem{Nagata} M. Nagata, \textit{Local Rings}, Interscience, New York, 1962.

\bibitem{P}     H. Pr\"ufer, \textit{Untersuchungen uber teilbarkeitseigenschaften in korpern}, J. Reine Angew. Math. 168 (1932) 1--36.\par
%\bibitem{Ra}    W. Rant, Minimally generated modules, Canad. Math. Bull. 23 (1980) 103--105.\par
%\bibitem{Ro}    J. J. Rotman, \textit{An Introduction to Homological Algebra}, Academic Press, New York, 1979.\par
%\bibitem{Sh}    R. Y. Sharp, Steps in Commutative Algebra. Second edition. London Mathematical Society Student Texts, 51. Cambridge University Press, Cambridge, 2000.\par
\bibitem{T}     H. Tsang, \textit{Gauss's Lemma}, Ph.D. thesis, University of Chicago, Chicago, 1965.\par
\end{thebibliography}

\end{document}